\documentclass[a4paper]{amsart}
\usepackage{amscd}
\usepackage{amsmath}
\usepackage{amssymb}
\usepackage{mathrsfs}
\usepackage{amsthm}
\usepackage{bbm} 
\usepackage{stmaryrd}
\usepackage{tikz-cd}

\usepackage[T1]{fontenc}

\makeatletter
\@namedef{subjclassname@2020}{%
  \textup{2020} Mathematics Subject Classification}
\makeatother

\numberwithin{equation}{section} \swapnumbers

\newtheorem{satz}{Satz}[section]

\newtheorem{theorem}[satz]{Theorem}
\newtheorem{proposition}[satz]{Proposition}

\newtheorem{lemma}[satz]{Lemma}

\newtheorem{remark}[satz]{Remark}

\newcommand{\bbr}{\mathbb{R}}
\newcommand{\bbe}{\mathbb{E}}
\newcommand{\bbn}{\mathbb{N}}
\newcommand{\bbp}{\mathbb{P}}

\newcommand{\calb}{\mathscr{B}}

\newcommand{\cale}{\mathscr{E}}
\newcommand{\calf}{\mathscr{F}}
\newcommand{\calg}{\mathscr{G}}

\newcommand{\calt}{\mathscr{T}}

\newcommand{\Id}{{\rm Id}}

\begin{document}

\title[Permutation invariant SLLN for exchangeable sequences]{Permutation invariant strong law of large numbers for exchangeable sequences}
\author{Stefan Tappe}
\address{Albert Ludwig University of Freiburg, Department of Mathematical Stochastics, Ernst-Zermelo-Stra\ss{}e 1, D-79104 Freiburg, Germany}
\email{stefan.tappe@math.uni-freiburg.de}
\date{5 November, 2021}
\thanks{I gratefully acknowledge financial support from the Deutsche Forschungsgemeinschaft (DFG, German Research Foundation) -- project number 444121509.}
\begin{abstract}
We provide a permutation invariant version of the strong law of large numbers for exchangeable sequences of random variables. The proof consists of a combination of the Koml\'{o}s-Berkes theorem, the usual strong law of large numbers for exchangeable sequences and de Finetti's theorem.
\end{abstract}

\maketitle\thispagestyle{empty}

\section{Introduction}

Kolmogorov's strong law of large numbers (SLLN) for independent and identically distributed (i.i.d.) sequences of random variables has been generalized into several directions. It has, for example, been generalized for pairwise independent, identically distributed random variables in \cite{Etemadi-1981}, for nonnegative random variables in \cite{Etemadi-1983}, for dependent, mixing random variables in \cite{Kuc-2005, Kuc-2011}, and for pairwise uncorrelated random variables in \cite{Janisch}.

There is also a version of the SLLN for exchangeable sequence. More precisely, let $(\xi_n)_{n \in \bbn}$ be an exchangeable sequence of random variables on a probability space $(\Omega,\calf,\bbp)$, let $\cale$ be its exchangeable $\sigma$-algebra, and let $\calt$ be its tail $\sigma$-algebra. If the sequence $(\xi_n)_{n \in \bbn}$ is integrable, then the SLLN for exchangeable sequences tells us that $(\xi_n)_{n \in \bbn}$ is almost surely Ces\`{a}ro convergent; more precisely, we have the following result.

\begin{proposition}\label{prop-SLLN-exch}
Let $(\xi_n)_{n \in \bbn} \subset L^1$ be an exchangeable sequence of integrable random variables. Then $( \xi_n )_{n \in \bbn}$ is $\bbp$-almost surely Ces\`{a}ro convergent to the limit
\begin{align*}
\xi = \bbe[\xi_1 | \cale] = \bbe[\xi_1 | \calt].
\end{align*}
\end{proposition}

This result is well-known; see, for example \cite[Example 12.15]{Klenke} or \cite[page 185]{Kingman}. The goal of this note is to establish the following permutation invariant version of the SLLN for exchangeable sequences.

\begin{theorem}\label{thm-main}
Let $(\xi_n)_{n \in \bbn} \subset L^1$ be an exchangeable sequence of integrable random variables. We set $\xi := \bbe[\xi_1 | \cale]$. Then the following statements are true:
\begin{enumerate}
\item For every subsequence $(n_k)_{k \in \bbn}$ and every permutation $\pi : \bbn \to \bbn$ the sequence $(\xi_{n_{\pi(k)}})_{k \in \bbn}$ is $\bbp$-almost surely Ces\`{a}ro convergent to $\xi$.

\item For every permutation $\sigma : \bbn \to \bbn$ and every subsequence $(m_k)_{k \in \bbn}$ the sequence $(\xi_{\sigma(m_k)})_{k \in \bbn}$ is $\bbp$-almost surely Ces\`{a}ro convergent to $\xi$.

\item We have $\bbp$-almost surely $\xi = \bbe[\xi_n | \cale] = \bbe[\xi_n | \calt]$ for each $n \in \bbn$.
\end{enumerate}
\end{theorem}

Intuitively, the statement of Theorem \ref{thm-main} is plausible. Indeed, de Finetti's theorem, which is stated as Theorem \ref{thm-de-Finetti} below, provides a connection between exchangeable sequences and conditional i.i.d. sequences, and in the present situation it implies that the sequence $(\xi_n)_{n \in \bbn}$ is i.i.d. given $\cale$ or given $\calt$.

Let us briefly indicate the main ideas for the proof of Theorem \ref{thm-main}. Since exchangeability of the sequence is preserved under permutations, by Proposition \ref{prop-SLLN-exch} it follows that the sequences $(\xi_{n_{\pi(k)}})_{k \in \bbn}$ and $(\xi_{\sigma(m_k)})_{k \in \bbn}$ are almost surely Ces\`{a}ro convergent. However, it is not clear whether the limits of these two sequences coincide with $\xi$, because their exchangeable $\sigma$-algebras can be different from $\cale$, and accordingly their tail $\sigma$-algebras can be different from $\calt$. Nevertheless, note that by exchangeability of the sequence all these limits have the same distribution.

In order to overcome the problem regarding the identification of the limits, we use the Koml\'{o}s-Berkes theorem (see \cite{Berkes}), which is stated as Theorem \ref{thm-Komlos-Berkes} below. This result is an extension of Koml\'{o}s's theorem (see \cite{Komlos}); see also \cite[Thm. 5.2.1]{Kabanov-Safarian} for another extension of Koml\'{o}s's theorem. The Koml\'{o}s-Berkes theorem was also used in order to prove the von Weizs\"{a}cker theorem (see \cite{W}); see also \cite[Thm. 5.2.3]{Kabanov-Safarian} for a similar result, and \cite{Tappe-W} for a note on the von Weizs\"{a}cker theorem. 

Coming back to the identification of the limits, the Koml\'{o}s-Berkes theorem provides us with a subsequence $(n_k)_{k \in \bbn}$ such that for every permutation $\pi : \bbn \to \bbn$ the sequence $(\xi_{n_{\pi(k)}})_{k \in \bbn}$ is almost surely Ces\`{a}ro convergent to the same limit. Using this result, in three steps we will show that for every subsequence and every permutation the corresponding sequence is almost surely Ces\`{a}ro convergent to the same limit, and that this limit is given by $\xi$. For the identification of the limits we use results about conditional expectations which are provided in Appendix~\ref{app-cond-exp}.

\section{Proof of the result}\label{sec-exchangeable}

Let $(\Omega,\calf,\bbp)$ be a probability space. We denote by $L^1 = L^1(\Omega,\calf,\bbp)$ the space of all equivalence classes of integrable random variables. Let $(\xi_n)_{n \in \bbn}$ be a sequence of random variables. Furthermore, let $\cale$ be the exchangeable $\sigma$-algebra of the sequence $(\xi_n)_{n \in \bbn}$, and let $\calt$ be the tail $\sigma$-algebra of the sequence $(\xi_n)_{n \in \bbn}$. We assume that the sequence $( \xi_n )_{n \in \bbn}$ is exchangeable; that is, for every finite permutation $\pi : \bbn \to \bbn$ we have
\begin{align*}
\bbp \circ \big( (\xi_n)_{n \in \bbn} \big) = \bbp \circ \big( (\xi_{\pi(n)})_{n \in \bbn} \big),
\end{align*}
or equivalently, for all $k \in \bbn$, all pairwise different $n_1,\ldots,n_k \in \bbn$ and all pairwise different $m_1,\ldots,m_k \in \bbn$ we have
\begin{align*}
\bbp \circ ( \xi_{n_1},\ldots,\xi_{n_k} ) = \bbp \circ (\xi_{m_1},\ldots,\xi_{m_k}).
\end{align*}

\begin{remark}\label{rem-exch}
Note that for every subsequence $(n_k)_{k \in \bbn}$ and every permutation $\pi : \bbn \to \bbn$ the sequence $(\xi_{n_{\pi(k)}})_{k \in \bbn}$ is also exchangeable. Accordingly, for every permutation $\sigma : \bbn \to \bbn$ and every subsequence $(m_k)_{k \in \bbn}$ the sequence $(\xi_{\sigma(m_k)})_{k \in \bbn}$ is also exchangeable.
\end{remark}

\begin{lemma}\label{lemma-permutations}
The following statements are true:
\begin{enumerate}
\item[(1)] For every subsequence $(n_k)_{k \in \bbn}$ and every permutation $\sigma : \bbn \to \bbn$ there exist a permutation $\pi : \bbn \to \bbn$ and a subsequence $(m_k)_{k \in \bbn}$ such that $\sigma(m_k) = n_{\pi(k)}$ for all $k \in \bbn$.

\item[(2)] For every subsequence $(n_k)_{k \in \bbn}$ and every permutation $\pi : \bbn \to \bbn$ there exists a permutation $\sigma : \bbn \to \bbn$ such that $\sigma(n_k) = n_{\pi(k)}$ for all $k \in \bbn$.
\end{enumerate}
\end{lemma}

\begin{proof}
(1) We define the one-to-one map $\tau : \bbn \to \bbn$ as
\begin{align*}
\tau(k) := \sigma^{-1}(n_k) \quad \text{for each $k \in \bbn$.} 
\end{align*}
Then there exists a permutation $\pi : \bbn \to \bbn$ such that $\tau(\pi(k)) < \tau(\pi(k+1))$ for all $k \in \bbn$. Indeed, we define $\pi$ inductively as follows. Let $\pi(1) \in \bbn$ be the unique index such that
\begin{align*}
\tau(\pi(1)) = \min \{ \tau(k) : k \in \bbn \}.
\end{align*}
If $\pi(1),\ldots,\pi(p)$ are already defined for some $p \in \bbn$, then let $\pi(p+1) \in \bbn$ be the unique index such that
\begin{align*}
\tau(\pi(p+1)) = \min \{ \tau(k) : k \in \bbn \setminus \{ \pi(1),\ldots,\pi(p) \} \}.
\end{align*}
Then $\pi$ is a permutation. We define the subsequence $(m_k)_{k \in \bbn}$ as $m_k := \tau(\pi(k))$ for each $k \in \bbn$. Then we have $\sigma(m_k) = n_{\pi(k)}$ for each $k \in \bbn$.

\noindent (2) We define the permutation $\sigma : \bbn \to \bbn$ as
\begin{align*}
\sigma(m) :=
\begin{cases}
n_{\pi(k)}, & \text{if $m = n_k$ for some $k \in \bbn$,}
\\ m, & \text{otherwise.}
\end{cases}
\end{align*}
Then we have $\sigma(n_k) = n_{\pi(k)}$ for all $k \in \bbn$.
\end{proof}

For convenience of the reader, we state the Koml\'{o}s-Berkes theorem and de Finetti's theorem, before we provide the proof of Theorem \ref{thm-main}.

\begin{theorem}[Koml\'{o}s-Berkes theorem]\label{thm-Komlos-Berkes}
Let $(\xi_n)_{n \in \bbn} \subset L^1$ be a sequence of integrable random variables such that $\sup_{n \in \bbn} \bbe[|\xi_n|] < \infty$. Then there exist a subsequence $(n_k)_{k \in \bbn}$ and an integrable random variable $\xi \in L^1$ such that for every permutation $\pi : \bbn \to \bbn$ the sequence $(\xi_{n_{\pi(k)}})_{k \in \bbn}$ is $\bbp$-almost surely Ces\`{a}ro convergent to $\xi$.
\end{theorem}

\begin{proof}
See \cite{Berkes}.
\end{proof}

Let $\calg \subset \calf$ be a sub $\sigma$-algebra. A sequence $(\xi_n)_{n \in \bbn}$ of random variables is called \emph{independent and identically distributed (i.i.d.) given $\calg$} if for every finite subset $I \subset \bbn$ and all Borel sets $B_i \in \calb(\bbr)$, $i \in I$ we have $\bbp$-almost surely
\begin{align*}
\bbp \bigg( \bigcap_{i \in I} \{ \xi_i \in B_i \} \bigg| \calg \bigg) = \prod_{i \in I} \bbp( \xi_i \in B_i | \calg ), \quad \text{(independence given $\calg$)}
\end{align*}
and for all $n,m \in \bbn$ and every Borel set $B \in \calb(\bbr)$ we have $\bbp$-almost surely
\begin{align*}
\bbp( \xi_n \in B | \calg ) = \bbp( \xi_m \in B | \calg ). \quad \text{(identical distributions given $\calg$)}
\end{align*}

\begin{theorem}[De Finetti's theorem]\label{thm-de-Finetti}
Let $(\xi_n)_{n \in \bbn}$ be a sequence of random variables. Then the following statements are equivalent:
\begin{enumerate}
\item[(i)] The sequence $(\xi_n)_{n \in \bbn}$ is exchangeable.

\item[(ii)] There exists a sub $\sigma$-algebra $\calg \subset \calf$ such that $(\xi_n)_{n \in \bbn}$ is i.i.d. given $\calg$.
\end{enumerate}
If the previous conditions are fulfilled, then we can choose $\calg = \cale$ or $\calg = \calt$.
\end{theorem}

\begin{proof}
See, for example \cite[Thm. 12.24]{Klenke}.
\end{proof}

Now, we are ready to provide the proof of Theorem \ref{thm-main}.

\begin{proof}[Proof of Theorem \ref{thm-main}]
By the Koml\'{o}s-Berkes theorem (see Theorem \ref{thm-Komlos-Berkes}) there exist a subsequence $(n_k)_{k \in \bbn}$ and an integrable random variable $\xi \in L^1$ such that for every permutation $\pi : \bbn \to \bbn$ the sequence $(\xi_{n_{\pi(k)}})_{k \in \bbn}$ is $\bbp$-almost surely Ces\`{a}ro convergent to $\xi$. Now, we proceed with the following three steps:

\noindent \emph{Step 1:} First, we show that for every permutation $\sigma : \bbn \to \bbn$ the sequence $(\xi_{\sigma(n)})_{n \in \bbn}$ is $\bbp$-almost surely Ces\`{a}ro convergent to $\xi$. Indeed, by Lemma \ref{lemma-permutations} there exist a permutation $\pi : \bbn \to \bbn$ and a subsequence $(m_k)_{k \in \bbn}$ such that $\sigma(m_k) = n_{\pi(k)}$ for each $k \in \bbn$. By Remark \ref{rem-exch} and Proposition \ref{prop-SLLN-exch} we have 
\begin{align}\label{xi-limit}
\xi = \bbe[\xi_{n_{\pi(1)}} \,|\, \cale_{(n_{\pi(k)})_{k \in \bbn}}] = \bbe[\xi_{n_{\pi(1)}} \,|\, \calt_{(n_{\pi(k)})_{k \in \bbn}}], 
\end{align}
where $\cale_{(n_{\pi(k)})_{k \in \bbn}}$ denotes the exchangeable $\sigma$-algebra of the sequence $(\xi_{n_{\pi(k)}})_{k \in \bbn}$, and $\calt_{(n_{\pi(k)})_{k \in \bbn}}$ denotes the tail $\sigma$-algebra of the sequence $(\xi_{n_{\pi(k)}})_{k \in \bbn}$. Furthermore, by Remark \ref{rem-exch} and Proposition \ref{prop-SLLN-exch} the sequence $(\xi_{\sigma(n)})_{n \in \bbn}$ is $\bbp$-almost surely Ces\`{a}ro convergent to the random variable
\begin{align}\label{eta-limit}
\eta := \bbe[\xi_{\sigma(1)} \,|\, \cale_{(\sigma(n))_{n \in \bbn}}] = \bbe[\xi_{\sigma(1)} \,|\, \calt_{(\sigma(n))_{n \in \bbn}}],
\end{align}
where $\cale_{(\sigma(n))_{n \in \bbn}}$ denotes the exchangeable $\sigma$-algebra of the sequence $(\xi_{\sigma(n)})_{n \in \bbn}$, and $\calt_{(\sigma(n))_{n \in \bbn}}$ denotes the tail $\sigma$-algebra of the sequence $(\xi_{\sigma(n)})_{n \in \bbn}$. By de Finetti's theorem (see Theorem \ref{thm-de-Finetti}) we have
\begin{align*}
\eta = \bbe[\xi_{n_{\pi(1)}} \,|\, \cale_{(\sigma(n))_{n \in \bbn}}] = \bbe[\xi_{n_{\pi(1)}} \,|\, \calt_{(\sigma(n))_{n \in \bbn}}].
\end{align*}
Since $\cale_{(n_{\pi(k)})_{k \in \bbn}} \subset \cale_{(\sigma(n))_{n \in \bbn}}$ and $\calt_{(n_{\pi(k)})_{k \in \bbn}} \subset \calt_{(\sigma(n))_{n \in \bbn}}$, by (\ref{xi-limit}) we obtain
\begin{align*}
\xi = \bbe[ \eta \,|\, \cale_{(n_{\pi(k)})_{k \in \bbn}}] = \bbe[ \eta \,|\, \calt_{(n_{\pi(k)})_{k \in \bbn}}].
\end{align*}
By exchangeability of the sequence $(\xi_n)_{n \in \bbn}$ we have
\begin{align*}
\bbp \circ \bigg( \frac{1}{n} \sum_{i=1}^n \xi_{n_{\pi(i)}} \bigg) = \bbp \circ \bigg( \frac{1}{n} \sum_{i=1}^n \xi_{\sigma(i)} \bigg) \quad \text{for each $n \in \bbn$,} 
\end{align*}
and hence, by Proposition \ref{prop-cond-exp-3} we obtain $\bbp$-almost surely $\xi = \eta$. In particular, if $\sigma = \Id$, then by (\ref{eta-limit}) and de Finetti's theorem (see Theorem \ref{thm-de-Finetti}) we obtain $\bbp$-almost surely
\begin{align*}
\xi = \eta = \bbe[\xi_n | \cale] = \bbe[\xi_n | \calt] \quad \text{for each $n \in \bbn$.}
\end{align*}

\noindent \emph{Step 2:} Now, let $\sigma : \bbn \to \bbn$ be an arbitrary permutation, and let $(m_k)_{k \in \bbn}$ be an arbitrary subsequence. Then the sequence $(\xi_{\sigma(m_k)})_{k \in \bbn}$ is $\bbp$-almost surely Ces\`{a}ro convergent to $\xi$. Indeed, by Step 1 and de Finetti's theorem (see Theorem \ref{thm-de-Finetti}) the sequence $(\xi_{\sigma(n)})_{n \in \bbn}$ is $\bbp$-almost surely Ces\`{a}ro convergent to
\begin{align}\label{xi-limit-2}
\xi = \bbe[\xi_{\sigma(1)} \,|\, \cale_{(\sigma(n))_{n \in \bbn}}] = \bbe[\xi_{\sigma(m_1)} \,|\, \cale_{(\sigma(n))_{n \in \bbn}}].
\end{align}
Furthermore, by Remark \ref{rem-exch} and Proposition \ref{prop-SLLN-exch} the sequence $(\xi_{\sigma(m_k)})_{k \in \bbn}$ is $\bbp$-almost surely Ces\`{a}ro convergent to the random variable
\begin{align*}
\zeta := \bbe[\xi_{\sigma(m_1)} \,|\, \cale_{(\sigma(m_k))_{k \in \bbn}}].
\end{align*}
Since $\cale_{(\sigma(m_k))_{k \in \bbn}} \subset \cale_{(\sigma(n))_{n \in \bbn}}$, by (\ref{xi-limit-2}) we obtain
\begin{align*}
\zeta = \bbe[\xi \,|\, \cale_{(\sigma(m_k))_{k \in \bbn}}].
\end{align*}
By exchangeability of the sequence $(\xi_n)_{n \in \bbn}$ we have
\begin{align*}
\bbp \circ \bigg( \frac{1}{n} \sum_{i=1}^n \xi_{\sigma(i)} \bigg) = \bbp \circ \bigg( \frac{1}{n} \sum_{i=1}^n \xi_{\sigma(m_i)} \bigg) \quad \text{for each $n \in \bbn$,} 
\end{align*}
and hence, by Proposition \ref{prop-cond-exp-3} we obtain $\bbp$-almost surely $\xi = \zeta$. Consequently, the sequence $(\xi_{\sigma(m_k)})_{k \in \bbn}$ is $\bbp$-almost surely Ces\`{a}ro convergent to $\xi$.

\noindent \emph{Step 3:} Now, let $(n_k)_{k \in \bbn}$ be an arbitrary subsequence, and let $\pi : \bbn \to \bbn$ be an arbitrary permutation. By Lemma \ref{lemma-permutations} there exists a permutation $\sigma : \bbn \to \bbn$ such that $\sigma(n_k) = n_{\pi(k)}$ for all $k \in \bbn$. Therefore, by Step 2 the sequence $(\xi_{n_{\pi(k)}})_{k \in \bbn}$ is $\bbp$-almost surely Ces\`{a}ro convergent to $\xi$, which concludes the proof.
\end{proof}

We can extend the statement of Theorem \ref{thm-main} as follows.

\begin{proposition}
Let $(\xi_n)_{n \in \bbn} \subset L^1$ be an exchangeable sequence of integrable random variables. We set $\xi := \bbe[\xi_1 | \cale]$. Then for every subsequence $(n_k)_{k \in \bbn}$ and all permutations $\pi,\sigma : \bbn \to \bbn$ the sequence $(\xi_{\sigma(n_{\pi(k)})})_{k \in \bbn}$ is $\bbp$-almost surely Ces\`{a}ro convergent to $\xi$. Furthermore, we have $\bbp$-almost surely $\xi = \bbe[\xi_n | \cale] = \bbe[\xi_n | \calt]$ for each $n \in \bbn$.
\end{proposition}

\begin{proof}
By Lemma \ref{lemma-permutations} there exists a permutation $\tau : \bbn \to \bbn$ such that $\tau(n_k) = n_{\pi(k)}$ for all $k \in \bbn$. The mapping $\rho : \bbn \to \bbn$ given by $\rho := \sigma \circ \tau$ is also a permutation, and we have $\sigma(n_{\pi(k)}) = \rho(n_k)$ for all $k \in \bbn$. Therefore, applying Theorem \ref{thm-main} concludes the proof.
\end{proof}

We conclude this section with the following consequence regarding Koml\'{o}s's theorem for exchangeable sequences. Namely, let $(\xi_n)_{n \in \bbn} \subset L^1$ be an exchangeable sequence of random variables. Then Theorem \ref{thm-main} shows that both extensions of Koml\'{o}s's theorem (the Koml\'{o}s-Berkes theorem from \cite{Berkes}, which we have stated as Theorem \ref{thm-Komlos-Berkes}, and \cite[Thm. 5.2.1]{Kabanov-Safarian}) are true with the original sequence $(\xi_n)_{n \in \bbn}$; that is, we do not have to pass to a subsequence $(\xi_{n_k})_{k \in \bbn}$.

\begin{appendix}

\section{Results about conditional expectations}\label{app-cond-exp}

We require the following results about conditional expectations. Since these results were not immediately available in the literature, we provide the proofs. For what follows, let $\calg \subset \calf$ be a sub $\sigma$-algebra.

\begin{lemma}\label{lemma-cond-exp-1}
Let $X \in L^2$ be a square-integrable random variable such that $\bbp \circ X = \bbp \circ \bbe[X | \calg]$. Then we have $\bbp$-almost surely $X = \bbe[X | \calg]$.
\end{lemma}

\begin{proof}
Setting $Y := \bbe[X | \calg]$, we have $\bbe[X^2] = \bbe[Y^2]$, and hence
\begin{align*}
\bbe[(X-Y)^2] &= \bbe[X^2] - 2 \bbe[XY] + \bbe[Y^2] = 2 \bbe[X^2] - 2 \bbe[\bbe[XY|\calg]]
\\ &= 2 \bbe[X^2] - 2 \bbe[Y \bbe[X|\calg]] = 2 \bbe[X^2] - 2 \bbe[Y^2] = 0,
\end{align*}
completing the proof.
\end{proof}

\begin{lemma}\label{lemma-Jensen-2}
Let $X \in L^1$ be a nonnegative random variable, and let $\varphi : \bbr^+ \to \bbr^+$ be a concave function such that $\bbp$-almost surely
\begin{align}\label{E-phi-equal}
\bbe[\varphi(X)] = \bbe[\varphi(\bbe[X | \calg])]. 
\end{align}
Then we have $\bbp$-almost surely
\begin{align}\label{E-phi-equal-2}
\bbe[\varphi(X) | \calg] = \varphi(\bbe[X | \calg]).
\end{align}
\end{lemma}

\begin{proof}
By Jensen's inequality for concave functions and conditional expectations we have $\bbp$-almost surely
\begin{align*}
\bbe[\varphi(X) | \calg] \leq \varphi(\bbe[X | \calg]).
\end{align*}
Suppose that (\ref{E-phi-equal-2}) does not hold true. Then we have $\bbp$-almost surely
\begin{align*}
\varphi(\bbe[X | \calg]) - \bbe[\varphi(X) | \calg] \in L_+^0 \setminus \{ 0 \},
\end{align*}
where $L_+^0$ denotes the convex cone of all equivalence classes of nonnegative random variables. Hence, we obtain $\bbp$-almost surely
\begin{align*}
\bbe[\varphi(X)] = \bbe[\bbe[\varphi(X) | \calg]] < \bbe [ \varphi(\bbe[X | \calg]) ],
\end{align*}
which contradicts (\ref{E-phi-equal}).
\end{proof}

\begin{lemma}\label{lemma-cond-exp-2}
Let $X \in L^1$ be an integrable random variable such that
\begin{align}\label{same-distr}
\bbp \circ X = \bbp \circ \bbe[X | \calg]. 
\end{align}
Then we have $\bbp$-almost surely $X = \bbe[X | \calg]$.
\end{lemma}

\begin{proof}
First, we assume that $X \in L^1$ is nonnegative. Let $n \in \bbn$ be arbitrary. By (\ref{same-distr}) and Lemma \ref{lemma-Jensen-2} we have $\bbp$-almost surely
\begin{align}\label{cond-exp-concave}
\bbe[X | \calg] \wedge n = \bbe[ X \wedge n | \calg ].
\end{align}
Therefore, by taking into account (\ref{same-distr}) we have 
\begin{align*}
\bbp \circ (X \wedge n) = \bbp \circ ( \bbe[X | \calg] \wedge n ) = \bbp \circ \bbe[X \wedge n | \calg]. 
\end{align*}
Since $X \wedge n \in L^2$, by Lemma \ref{lemma-cond-exp-1} and (\ref{cond-exp-concave}) we deduce that $\bbp$-almost surely
\begin{align*}
X \wedge n = \bbe[ X \wedge n | \calg ] = \bbe[X | \calg] \wedge n.
\end{align*}
Since $n \in \bbn$ was arbitrary, it follows that $\bbp$-almost surely $X=\bbe[X | \calg]$.

Now, let $X \in L^1$ be arbitrary. Since
\begin{align*}
\bbe[X|\calg]^+ = \bbe[X^+|\calg] \quad \text{and} \quad \bbe[X|\calg]^- = \bbe[X^-|\calg],
\end{align*}
by (\ref{same-distr}) we have
\begin{align*}
\bbp \circ X^+ = \bbp \circ \bbe[X^+ | \calg] \quad \text{and} \quad \bbp \circ X^- = \bbp \circ \bbe[X^- | \calg].
\end{align*}
By the first part of the proof, we deduce that $\bbp$-almost surely $X^+ = \bbe[X^+ | \calg]$ and $X^- = \bbe[X^- | \calg]$, and hence $X = \bbe[X | \calg]$.
\end{proof}

\begin{proposition}\label{prop-cond-exp-3}
Let $(X_n)_{n \in \bbn}$ and $(Y_n)_{n \in \bbn}$ be sequences of random variables, and let $X \in L^1$ be an integrable random variable. We assume that $\bbp \circ X_n = \bbp \circ Y_n$ for each $n \in \bbn$, and that $X_n \overset{\text{a.s.}}{\to} X$ and $Y_n \overset{\text{a.s.}}{\to} \bbe[X|\calg]$ as $n \to \infty$. Then we have $\bbp$-almost surely $X = \bbe[X | \calg]$.
\end{proposition}

\begin{proof}
Noting that $\bbp \circ X = \bbp \circ \bbe[X | \calg]$, this is a consequence of Lemma \ref{lemma-cond-exp-2}.
\end{proof}

\end{appendix}


\begin{thebibliography}{20}

\bibitem{Berkes} Berkes, I. (1990):
  An extension of the Koml\'{o}s subsequence theorem.
  \textit{Acta Math. Hungar.} {\bf 55}(1--2), 103--110.   

\bibitem{Etemadi-1981} Etemadi, N. (1981):
  An elementary proof of the strong law of large numbers.
  \textit{Z. Wahrscheinlichkeitstheorie verw. Gebiete} {\bf 55}(1), 119--122.  
  
\bibitem{Etemadi-1983} Etemadi, N. (1983):
  On the laws of large numbers for nonnegative random variables.
  \textit{J. Multivar. Anal.} {\bf 13}(1), 187--193.   

\bibitem{Janisch} Janisch, M. (2021):
  Kolmogorov's strong law of large numbers holds for pairwise uncorrelated random variables.
  arXiv: 2005.03967v2.  
  \textit{Theory Probab. Appl.} {\bf 66}(2), 263--275.   
  
\bibitem{Kabanov-Safarian} Kabanov, Y.~M., Safarian, M. (2009):
  \textit{Markets with Transaction Costs. Mathematical Theory}. Springer, Berlin.  
  
\bibitem{Kingman} Kingman, J.~F.~C. (1978):
  Uses of exchangeability.
  \textit{Ann. Probab.} {\bf 6}(2), 183--197.
    
\bibitem{Klenke} Klenke, A. (2014):
  \textit{Probability Theory. A Comprehensive Course}. Second Edition, Springer, London.  
    
\bibitem{Komlos} Koml\'{o}s, J. (1967):
  A generalization of a problem of Steinhaus.
  \textit{Acta Math. Sci. Hung} {\bf 18}(1--2), 217--229.   
  
\bibitem{Kuc-2005} Kuczmaszewska, A. (2005):
  The strong law of large numbers for dependent random variables.
  \textit{Stat. Probab. Lett.} {\bf 73}(3), 305--314.  
  
\bibitem{Kuc-2011} Kuczmaszewska, A. (2011):
  On the strong law of large numbers for $\phi$-mixing and $\rho$-mixing random variables.
  \textit{Acta Math. Hungar.} {\bf 132}(1--2), 174--189.  
  
\bibitem{Tappe} Tappe, S. (2020):
  Permutation invariant strong law of large numbers for exchangeable sequences.
  arXiv: 2009.08288v1.
  
\bibitem{Tappe-W} Tappe, S. (2021): 
  A note on the von Weizs\"{a}cker theorem. 
  \textit{Stat. Probab. Lett.} {\bf 168}, Article 108926, 6 pages.  
  
\bibitem{W} von Weizs\"{a}cker, H. (2004):
  Can one drop $L^1$-boundedness in Koml\'{o}s subsequence theorem?
  \textit{Amer. Math. Monthly} {\bf 111}(10), 900--903.  
  
\end{thebibliography}
\end{document}